\documentclass[12pt,reqno]{amsart}

\usepackage[T1]{fontenc}
\pdfoutput=1
\usepackage{mathptmx,microtype,graphicx,dsfont}
\usepackage[hidelinks]{hyperref}
\usepackage{amsthm,amsmath,amssymb}
\usepackage{enumitem}
\usepackage[nameinlink,capitalise]{cleveref}
\usepackage[font=footnotesize,margin=5em]{caption}
\usepackage{cite}

\usepackage[foot]{amsaddr}

\theoremstyle{plain}

\newtheorem{theorem}{Theorem}[section]
\newtheorem{lemma}[theorem]{Lemma}

\newtheorem{corollary}[theorem]{Corollary}

\newcommand\cA{\mathcal{A}}
\newcommand{\cB}{{\mathcal B}}

\newcommand\cS{\mathcal{S}}

\newcommand{\cX}{{\mathcal X}}

\newcommand{\bfx}{{\bf x}}
\newcommand{\bfy}{{\bf y}}

\newcommand\Z{\mathbb{Z}}
\newcommand\R{\mathbb{R}}

\renewcommand{\P}{{\bf P}}

\renewcommand{\le}{\leqslant}
\renewcommand{\ge}{\geqslant}

\author[S. Donderwinkel]{Serte Donderwinkel}
\address{University of Groningen, 
Bernoulli Institute for Mathematics, 
Computer Science and AI, 
and CogniGron (Groningen Cognitive Systems and Materials Center)}
\email{s.a.donderwinkel@rug.nl}

\author[B. Kolesnik]{Brett Kolesnik}
\address{University of Warwick, Department of Statistics}
\email{brett.kolesnik@warwick.ac.uk}

\keywords{entropic repulsion; 
infinite divisibility;
majorization; 
persistence probability; 
random polymer;
random walk;
renewal sequence;
wetting model}
\subjclass[2010]{05A15;	
05A17; 	
05C20;	
60E07;	
60G50;	
60K05;	
82B41; 	
82D60}	

\begin{document}

\title[Sina\u{\i} excursions]
{Sina\u{\i} excursions: 
An analogue of Sparre Andersen's formula for the area process of a random walk}

\begin{abstract} 
Sina\u{\i} initiated the study of 
random walks with 
persistently positive area processes, 
motivated by 
shock waves in solutions to the 
inviscid Burgers' equation. 
We find the 
precise asymptotic probability 
that the area process 
of a random walk bridge is an excursion. 
A key ingredient is an analogue of 
Sparre Andersen's classical formula. 
The asymptotics are related to 
von Sterneck's subset counting formulas 
from additive number theory. 
Our results 
sharpen bounds by 
Aurzada, Dereich and Lifshits
and 
respond to a question of 
Caravenna and Deuschel, 
which arose in their study of the wetting model. 
In this context, Sina\u{\i} excursions
are a class of random polymer chains 
exhibiting entropic repulsion. 
\end{abstract}

\maketitle

\section{Introduction}\label{S_intro}

\subsection{Persistence
probabilities}

There is a rich history 
of calculating {\it persistence
probabilities} in mathematics, wherein
we ask for a random process to 
continue to satisfy some property of interest.
For instance, 
Bertrand's ballot theorem \cite{ABR08} from 1887 can be 
viewed as an early example. 
No less famous is 
Sparre~Andersen's \cite{SpaAnd54}
formula from 1954, concerning the probability 
that a simple random walk remains positive. 

Applications of persistence in statistical physics began
in the late 1980s, as it relates to, e.g., fluctuating interfaces 
and sticky particle systems. We refer 
to the surveys  \cite{Maj99,CDPCMDS04,BMS13,AS15}
for a detailed overview of existing results and applications. 

In this work, we introduce 
an analogue of Sparre~Andersen's formula 
for the {\it area process} of a random walk; see \eqref{E_SA54} and 
Theorem \ref{T_SAforSE} below. 

Persistence probabilities for such 
area processes were first studied by Sina\u{\i} \cite{Sin92}, 
in relation to   
the inviscid Burgers' equation. 
This equation models
a turbulent fluid  
which gives rise to {\it shock waves}.
Such discontinuities are overcome  mathematically 
via {\it control surfaces,} i.e., interfaces  through which material flows from one side of a discontinuity to the other. 
When the system is started with self-similar, 
Brownian 
data, the exponent $1/4$ in 
\eqref{E_Sin92} below 
is related to the fact that the 
set of initial positions of 
particles not yet ``shocked'' by time $t=1$
has Hausdorff dimension $1/2$; see 
\cite{AS15,Sin92b,Ber98,Han93}. 

Times when the random walk and its area
process revisit 0 are renewal times
for such persistence problems. 
We call the trajectories between such times
{\it Sina\u{\i} excursions;} see 
Section \ref{S_SEs} below. 
These excursions play a key role 
in the work of 
Caravenna and Deuschel \cite{CD08}
on the {\it wetting model}. 
In this context, Sina\u{\i} excursions are related to random polymer
chains exhibiting a phenomenon referred to as {\it entropic repulsion}. 
More specifically, the area process of a Sina\u{\i} excursion
is used to model the 
interface that forms 
between a gas pressed 
diffusively 
against a surface 
by a liquid;
see \cite{CD08,AS15,Bol00,Vel06,Gia07} for more 
details. 

Our main result Theorem \ref{T_main}
(see also Corollary \ref{cor:meander})  
identifies the precise asymptototic
probability that a random walk is a 
Sina\u{\i} excursion. This 
sharpens the original bounds (up to polylogarithmic factors) 
by Caravenna and Deuschel \cite{CD08}
and the subsequent improvement (up to constant factors) by 
Aurzada, Dereich and Lifshits\cite{ADL14}. 
Our proof utilizes 
our area analogue of Sparre~Andersen's formula mentioned above, 
old formulas from additive number theory proved in the early 
1900s by von Sterneck \cite{Bac02}, 
and the Tauberian 
theorems related to the L\'evy--Khintchine formula proved by
Hawkes and Jenkins \cite{HJ78}.

\subsection{Sina\u{\i} walks}

Let $(S_k:k\ge0)$ be a 
simple symmetric random walk 
on the integers $\Z$
started at $S_0=0$. 
We let 
$A_k=\sum_{i=1}^k S_i$ denote its 
cumulative area after $k$ steps. 
Sina\u{\i} \cite{Sin92} proved that 
\begin{equation}\label{E_Sin92}
\P(A_1,\ldots,A_{n}\ge0)=\Theta(n^{-1/4}).
\end{equation}

If $A_1,\ldots,A_{n}\ge 0$ holds, 
we call $(S_0,S_1,\ldots,S_n)$
a {\it Sina\u{\i} walk}. 

As discussed by 
Aurzada and Simon \cite[Section 3]{AS15}, 
in their survey
on {\it persistence probabilities,}
Sina\u{\i}'s proof is based on 
the sequence of times 
$0=\tau_0,\tau_1\ldots$ that the walk visits 0. 
This gives
rise to another random walk, whose increments 
are the signed areas accumulated 
between these times. 
A key ingredient  
is Sparre Andersen's \cite{SpaAnd54}
classical result that, for $|x|\le 1$, 
\begin{equation}\label{E_SA54}
\sum_{n=0}^\infty \P(T_0>n) x^n 
=
\exp\left(\sum_{k=1}^\infty \P(S_k\le 0)\frac{x^k}{k}\right), 
\end{equation}
where $T_0=\inf\{t:S_t>0\}$
is the first time that the walk is positive. 

The utility of \eqref{E_SA54} lies in the fact that 
the probabilities 
$\P(S_k\le 0)$ are simpler than 
$\P(T_0>n)$.
A Tauberian theorem 
implies  
$\P(T_0>\tau_n)\sim c n^{-1/2}$. 
Finally, \eqref{E_Sin92} 
is derived using that
$n^{-2}\tau_n$ converges
to a stable random variable.

We note that 
Vysotsky \cite{Vys14} sharpened
and generalized 
Sina\u{\i}'s persistence   probability
\eqref{E_Sin92}, showing that 
\[
\P(A_1,\ldots,A_{n}\ge0)\sim Cn^{-1/4},
\]
for a wide class of random walks. 
The constant $C$, however, is expressed
in terms of a rather complicated integral
(see equation (37) therein). 
Perhaps our current arguments can help 
with finding a more explicit, combinatorial 
description of $C$, at least in some cases. 

More specifically, in \cite{Vys14} it is assumed that
the increments of the random walk are $\alpha$-stable, 
for some $1<\alpha\le 2$, 
and either right-exponential
or right-continuous (see \cite{Vys14} for the precise conditions). 
Subsequently, 
Dembo, Ding and Gao
\cite{BDW23} extended these results
to the case that the increments have finite, positive second moment. See also
Gao, Liu and Yang \cite{BDW23}
for the case of Gaussian increments. 

Finally, we note that continuous 
analogues
of such persistence problems have been studied very recently
by 
B\"{a}r, Duraj and Wachtel \cite{BDW23},
giving rise to a class of random processes
called
{\it Kolmogorov diffusions}.

\begin{figure}[h]
\centering
\includegraphics[scale=1]{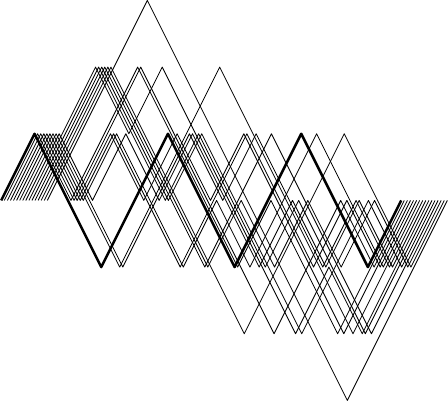}
\caption{The $16$ Sina\u{\i} 
excursions of length 
$12$, slightly staggered. 
Such excursions have total area $0$ and
non-negative partial areas. 
The standard Sina\u{\i} excursion in bold oscillates
between 
$\pm1$.
}
\label{F_SEs}
\end{figure}

\subsection{Sina\u{\i} excursions}
\label{S_SEs}

We call $(S_0,S_1,\ldots,S_{4n})$
a 
{\it Sina\u{\i} excursion} if it is a 
Sina\u{\i} walk and 
$S_{4n}=A_{4n}=0$.
Such {\it renewal times} are only possible
at multiples of $4$. 
We note that 
Sina\u{\i} excursions
are the discrete analogue of the    
{\it positive Kolmogorov excursions 
(from zero and back)}
studied in \cite{BDW23}.

Aurzada, Dereich and Lifshits  \cite{ADL14}
showed that 
\begin{equation}\label{E_ADL14}
p_n=\P(A_1,\ldots,A_{4n}\ge0\mid S_{4n}=A_{4n}=0)
=\Theta(n^{-1/2}),
\end{equation}
verifying a conjecture in Caravenna and Deuschel \cite{CD08}.

Our main result identifies the precise asymptotics. 

\begin{theorem}
\label{T_main}
As $n\to\infty$, 
\[
n^{1/2}p_n
\to 
\frac{1}{2}\sqrt{\frac{\pi}{6}}
\exp\left(\sum_{k=1}^\infty \frac{\Xi_k}{k2^{4k}}\right),
\]
where
\begin{equation}\label{E_vS}
\Xi_k=\frac{1}{4n}\sum_{d|2k}{4k/d-1\choose 2k/d}
\phi(d),
\end{equation}
and $\phi$ is Euler's totient function. 
\end{theorem}

Our methods also lead to the precise asymptotics of 
\begin{equation}\label{E_SinaiM}
\P(A_1,\ldots,A_{2n}\ge0\mid S_{2n}=0),
\end{equation}
corresponding to
a class of 
{\it Sina\u{\i} meanders}. 
See \cref{cor:meander}
below.

As we will discuss in Section \ref{S_vonS} below, 
$\Xi_n$ is the number of subsets 
of $\{1,2,\ldots,4n-1\}$ of size $2n$ that sum to $3n$
mod $4n$. This formula is 
one instance in a family of 
general modular subset
counting formulas proved by von Sterneck
in the early 1900s. 

To give a first hint about the connection 
between $p_n$ and $\Xi_n$, 
consider the $2n$ times  $t_1<\cdots<t_{2n}$ 
before down steps
in a 
walk $(S_0,S_1,\ldots,S_{4n})$ when 
$S_{t+1}-S_t=-1$. 
By Lemma \ref{L_area} below, if 
$A_{4n}=0$ then  
$\sum_{j=1}^{2n} t_j=n(4n-1)$. Thus, $\{t_1,\ldots,t_{2n} \}$ 
is a subset of $\{1,2,\ldots,4n-1\}$ of size $2n$ that sums to $3n$ mod $4n$.
More to the point, as we will see, 
$\Xi_n$ is related to the number of bridges that 
can be turned into a 
Sina\u{\i} excursion, by 
cyclically shifting their increments.

\subsection{A Sparre Andersen analogue}
To prove 
Theorem \ref{T_main}, we will first establish
the following 
analogue
of Sparre Andersen's formula \eqref{E_SA54}
for the probabilities 
\begin{equation}
\varphi_n
=\P(A_1,\ldots,A_{4n}\ge0,\: A_{4n}=S_{4n}=0). 
\end{equation}
By the local limit in \cite[Proposition 2.1]{ADL14}, 
it follows that 
\[
n^2\P(A_{4n}=S_{4n}=0)\to \frac{\sqrt{3}}{4\pi}. 
\]
Therefore, to prove Theorem \ref{T_main}, 
it suffices to show that
\begin{equation}
\label{E_wts}
n^{5/2}\varphi_n\to \frac{e^\lambda}{8\sqrt{2\pi}},
\end{equation}
where
\begin{equation}
\label{E_lam}
\lambda=\sum_{k=1}^\infty \frac{\Xi_k}{k2^{4k}}.
\end{equation}

For convenience, we put $\xi_n=\Xi_n/2^{4n}$
and let $\Phi_n=2^{4n}\varphi_n$ denote the number of
Sina\u{\i} excursions of length $4n$. 

\begin{theorem}
\label{T_SAforSE}
For $|x|\le1$, 
\begin{equation}\label{E_SA}
\sum_{n=0}^\infty \varphi_n x^n=
\exp\left(
\sum_{k=1}^\infty \xi_k \frac{x^k}{k}
\right).
\end{equation}
\end{theorem}

As with \eqref{E_SA54}, the usefulness of 
\eqref{E_SA} is that it allows for an indirect
analysis of the probabilities of interest 
$\varphi_n$. The criteria for  Sina\u{\i} excursions  
imposes conditions at all times along the trajectory. 
On the other hand, $\xi_n$
has a much 
simpler description, 
by von Sterneck's  
formulas (Lemma \ref{L_vonS} below).

\subsection{Transferring asymptotics}

To transfer asymptotic information 
from $\xi_n$ to $\varphi_n$, 
we will use a connection 
with {\it L\'evy processes}
$(L_t,t\ge0)$. 
We recall that a probability measure $\pi$ 
is infinitely divisible
if for all $m\ge1$ there are independent
and identically distributed $X_1,\ldots,X_m$
for which $X_1+\cdots+X_m\sim \pi$. 
L\'evy processes have  
independent, stationary increments, 
so the distribution of $L_t$ is infinitely divisible, 
at any given 
time $t>0$. 

The 
{\it L\'evy--Khintchine} formula 
relates an infinitely divisible 
$\pi$ to a certain   
{\it L\'evy measure} $\nu$, 
which controls the jumps in the associated 
L\'evy process $(L_t,t\ge0)$ such that $L_1\sim \pi$. 
In the case that 
$\pi=(p_n,n\ge0)$ is supported 
on the non-negative integers, 
we have that 
\begin{equation}\label{E_LK}
\sum_{n=0}^\infty p_n x^n=
\exp\left(
\sum_{k=1}^\infty (1-x^k)\nu_k, 
\right)
\end{equation}
where $(\nu_k,k\ge1)$
has finite total L\'evy measure
$\lambda=\sum_{k=1}^\infty \nu_k$. 
Furthermore, $\nu_k$ is the expected number of 
jumps of size $k$ by time $t=1$. 

Hence, by Theorem \ref{T_SAforSE}, it follows that 
$p_n=e^{-\lambda}\varphi_n$ is infinitely divisible, 
with corresponding L\'evy measure $\nu_n=\xi_n/n$. 
The constant $\lambda$ in \eqref{E_lam}, 
which appears in Theorem \ref{T_main}, 
is the total L\'evy measure. 
To complete the proof of Theorem \ref{T_main}, given 
Theorem \ref{T_SAforSE}, we will use the following result by 
Embrechts and Hawkes \cite{EH82}, 
which shows that $p_n\sim\nu_n$,
when $\nu_n$ is sufficiently regular. 

A probability distribution $(q_n,n\ge0)$ is 
{\it sub-exponential}
if $q_n/q_{n+1}\to1$ and its convolution
\[
q_n^*=\sum_{k=0}^n q_kq_{n-k}. 
\]
satisfies that  $q_n^*/q_n\to2$.
In \cite{EH82} it is proved that 
if $p_n$ and $\nu_n$ are related by \eqref{E_LK}, 
then $\nu_n/\lambda$ is sub-exponential if and only if 
$p_n\sim\nu_n$ and $\nu_n/\nu_{n+1}\to1$. 
Intuitively, this follows by the ``one big jump principle.''
A large value of $L_1$ is likely 
due to one big jump of essentially this value by time $t=1$. 
As a special case, 
Hawkes and Jenkins \cite{HJ78} 
showed that $p_n\sim\nu_n$ if $\nu_n$ 
is {\it regularly varying} with 
index $\gamma<-1$. 

Let us note that the essential 
feature of $\exp(z)$ in \eqref{E_LK}
is that it is analytic. 
Indeed, the results in 
\cite{EH82} are based on the work of 
Chover, Ney and Wainger \cite{CNW73} 
on analytic transformations of 
probability measures, and so extend 
to other analytic $f(z)$;
see Embrechts and Omey \cite{EO84}.

In the present case, 
by Stirling's approximation, 
\[
n^{5/2}\nu_n
=n^{3/2}\xi_n
=\frac{n^{3/2}}{2^{4n}}\Xi_n
\to \frac{1}{8\sqrt{2\pi}},
\]
since the term $d=1$ 
dominates in \eqref{E_vS}.
As such, $\nu_n$ is 
regularly varying with index $\gamma=-5/2$, 
and it follows that 
\[
n^{5/2}\varphi_n
=e^\lambda n^{5/2}p_n
\sim e^\lambda n^{5/2}\nu_n
\to \frac{e^\lambda}{8\sqrt{2\pi}},
\]
yielding \eqref{E_wts}. As discussed, 
Theorem \ref{T_main} follows.

\subsection{Outline}

We have shown how  
Theorem \ref{T_main} follows from 
Theorem \ref{T_SAforSE} and the asymptotics of $\Xi_n$.
The remainder of the article 
is devoted to the proof of 
Theorem \ref{T_SAforSE}.

\section{von Sterneck's formulas}
\label{S_vonS}

In the early 1900s, von Sterneck
(see, e.g., \cite{Bac02,Ram44})
found the 
number $\Lambda_k(n,s)$ 
of multi-sets 
$\{m_1,\ldots,m_k\}$ of $\{0,1,\ldots,n-1\}$ 
of size $k$ that sum to 
$\sum_{i=1}^k m_i\equiv s$ mod $n$. 

\begin{lemma}[von Sterneck]
\label{L_vonS}
For all $n\ge1$, we have that 
\[
\Lambda_k(n,s)
=\frac{1}{n}\sum_{d|k,n}{(n+k)/d-1\choose k/d}
\frac{\mu(d/\gcd(d,s)) \phi(d) }{\phi(d/\gcd(d,s))},
\]
where  
$\mu$ is the M\"obius function, 
$\phi$ is the Euler totient function,
and $\gcd(d,s)$ is the greatest common divisor 
of $d$ and $s$.
\end{lemma}

In particular, 
\[
\Lambda_{2n}(2n,0)
=\frac{1}{n}\sum_{d|2n}{4n/d-1\choose 2n/d}\phi(d).
\]
Therefore, to justify 
\eqref{E_vS} above, we prove the 
following claim.

\begin{lemma} 
\label{L_nor3n}
For all $n\ge1$, 
we have that 
\[
2\Xi_n=
\Lambda_{2n}(2n,0).
\]
\end{lemma}

\begin{proof}
First, we note that, for integers
$1\le a_1<\cdots<a_{2n}\le 4n-1$, 
we have 
$\sum_{i=1}^{2n} a_i\equiv n$ mod $4n$ if and only if 
$\sum_{i=1}^{2n} (4n-a_i)\equiv 3n$ mod $4n$. 
Hence, there is the same number of 
subsets of $\{1,2,\ldots,4n-1\}$
of size $2n$ that sum to $n$
mod $4n$ as there are that sum to $3n$
mod $4n$. 

Next, we claim that, 
for integers 
$0\le m_1\le\cdots\le m_{2n}\le 2n-1$, 
we have that $\sum_{i=1}^{2n} m_i$ 
is equal to $0$ mod $2n$
if and only if 
$\sum_{i=1}^{2n} (m_i+i)$ is equal to 
$n$ or $3n$ mod $4n$. 
Indeed, to see this, 
simply note that 
$
\sum_{i=1}^{2n}i=n(2n+1) 
$
is equal to $n$ or $3n$
mod $4n$ (depending on the parity of $n$). 
This implies that the sub-multisets of $\{1,\dots, 2n-1\}$ 
of size $2n$ that sum to $0$ mod $2n$ are in bijection 
with the subsets of $\{1,\dots, 4n-1\}$ of size 
$2n$ that sum to $n$ or $3n$ mod $n$, and the statement follows. 
\end{proof}

\section{Times before down steps}

Consider a bridge $\cB=(B_0,B_1,\ldots,B_{2n})$ 
of length $2n$.
That is, $B_0=B_{2n}=0$ and 
all increments 
$\Delta B_k=B_{k+1}-B_k=\pm1$, 
for $0\le k\le2n-1$. 
Let 
\[
{\bf t}(\cB)=(t_1,\ldots,t_{n})
\]
denote the sequences 
of {times before down steps,}
that is, times $0\le  t_1<\dots<t_n\le 2n-1$ 
such that 
$\Delta B_t=-1$. 

\begin{lemma}
\label{L_area}
Let $\cB=(B_0,B_1,\ldots,B_{4n})$  be a bridge
with times ${\bf t}(\cB)=(t_1,\ldots,t_{2n})$ 
before down steps. 
Then its total area 
\begin{equation}\label{E_area}
A_{4n}=\sum_{k=1}^{4n}B_k
=-2n(4n-1)+2\sum_{i=1}^{2n}t_i.
\end{equation}
\end{lemma}

\begin{proof}
To see this, note that 
\[
\sum_{k=1}^{4n}B_k
=
\sum_{k=0}^{4n-1}(4n-k)\Delta B_k
=\sum_{k=0}^{4n-1}(4n-k)
-2\sum_{j=1}^{2n}(4n-t_j),
\]
where the last step follows from $\Delta B_k=1-2\mathbf{1}{\{k\in \{t_1,\dots,t_{2n}\}\}}$.
This simplifies to 
$-2n(4n-1)+2\sum_{i=1}^{2n}t_i$,
as claimed. 
\end{proof}

We let  
\[
\cS=(0,1,0,-1,0\ldots,0,1,0,-1,0)
\]
denote the {\it standard Sina\u{\i} excursion}
of length $4n$; see Figure \ref{F_SEs}. 
Note that $\cS$
is a ``sawtooth'' bridge,  oscillating 
between $\pm1$, with 
\begin{equation}\label{E_tstar}
{\bf t}(\cS)
=(1,2,5,6,\ldots,4n-3,4n-2).
\end{equation}

Note that ${\bf t}(\cS)$ sums to 
$n(4n-1)$. Therefore, if $\cB$ is a 
Sina\u{\i} excursion then, by 
Lemma \ref{L_area}, 
${\bf t}(\cB)$ and ${\bf t}(\cS)$
have the same sum. 

In fact, it can be shown that 
$\cB$ is a 
Sina\u{\i} excursion if and only if 
${\bf t}(\cB)$ is {\it majorized} 
by ${\bf t}(\cS)$. 
(For weakly increasing $\bfx,\bfy\in\R^n$, 
$\bfx$ is majorized by $\bfy$
if all partial sums 
$\sum_{i=1}^k x_i\ge \sum_{i=1}^k y_i$, 
with equality when $k=n$.)
Intuitively, $\cS$ takes its down steps
as soon as possible, maintaining a cumulative
area $\ge0$ as close to 0 as possible. 
We will not require this fact, and omit
the details.

\section{Sparre Andersen for Sina\u{\i} excursions}
\label{E_SAforSE}

In this section, we prove 
Theorem \ref{T_SAforSE}, 
which states that 
\[
\sum_{n=0}^\infty \varphi_n x^n=
\exp\left(
\sum_{k=1}^\infty \xi_k \frac{x^k}{k}
\right).
\]
By differentiating and comparing coefficients, 
it can be seen that this is equivalent to 
\[
n\varphi_n=\sum_{k=1}^n \xi_k \varphi_{n-k}.
\]
Therefore, multiplying by $2^{4n}$, 
to prove Theorem \ref{T_SAforSE}  
it suffices to
show 
\begin{equation}\label{E_PhiXi_rec}
n\Phi_n=\sum_{k=1}^n \Xi_k \Phi_{n-k}.
\end{equation}

For this, we use the observation that $\Phi_n$ is a  
\emph{renewal sequence}. A renewal sequence enumerates 
structures that can be decomposed into a series of irreducible parts. 
Formally, $A_n$ is a renewal sequence if its generating function 
$A(x)=\sum_{n=0}^\infty A_n x^n$ can be expressed as 
\[A(x)=\sum_{m=0}^\infty [A^{(1)}(x)]^m=\frac{1}{1-A^{(1)}(x)}\]
where $A^{(1)}(x)=\sum_{n=1}^\infty A_n^{(1)}x^n$ is the generating 
function for the number $A_n^{(1)}$ of irreducible structures of length $n$. 
See, e.g., Feller \cite{F68} for details. 

We see that $\Phi_n$ is a renewal sequence because any
Sina\u{\i} excursion can 
be decomposed into a series of  \emph{ irreducible Sina\u{\i} excursions} 
whose area process only takes the value $0$ at the first and the last step.
More specifically, 
let $\Phi_n^{(1)}$ be the number of
walks $(S_0,S_1,\ldots,S_{4n})$ 
for which $S_0=A_0=0$, $S_{4n}=A_{4n}=0$
and $A_1,\ldots,A_{4n-1}>0$. 
Then, the generating function $\Phi(x)=\sum_n \Phi_n x^n$ of $\Phi_n$ 
can be expressed as   
\[
\Phi(x)
=\frac{1}{1-\Phi^{(1)}(x)},
\]
where  
$\Phi^{(1)}(x)=\sum_n \Phi_n^{(1)} x^n$ is the generating function 
for the number $\Phi^{(1)}_n$ of irreducible Sina\u{\i} excursions.

Then, the key to proving \eqref{E_PhiXi_rec}
is the following lemma, proved by Bassan and the authors in \cite{BDK25}.

\begin{lemma}[Lemma 2 (1) in \cite{BDK25}] \label{lem:renewal_mark}
Suppose that $1=A_0, A_1,\ldots$ 
is a renewal sequence, 
where $A_n$ counts the number of objects in a
class $\cA_n$ of objects of size $n$. Then 
\begin{equation}\label{E_ren_rec}
nA_n = \sum_{k=1}^n A_k' A_{n-k}, 
\end{equation}
where $A_n'$ is the number of pairs $(X,s)$,
where $X\in\cA_n$ 
and $0\le s<\ell$ is an integer, where $\ell$ 
is the size of the first irreducible part of $X$.
\end{lemma}
Its proof is simple, so we include it here. 
\begin{proof}
Note that $nA_n$ enumerates the pairs $(X,i)$,
where $X\in\cA_n$ 
and $0\le i<n$ is an integer. We call such pairs 
\emph{marked objects} of size $n$.
For a marked object $(X,i)$ of size $n$, consider the sub-object 
consisting of the irreducible part
containing $i$ and all subsequent parts.  
This sub-object is of size $k$, for some $1\le k \le n$, 
and has a mark somewhere in its first irreducible part. 
All other previous parts form an (unmarked) 
object of size $n-k$. This implies that there are 
$\sum_{k=1}^n A_k' A_{n-k} $
many such pairs $(X,i)$. 
See Figure \ref{F_Aprime}. 
\end{proof}

\begin{figure}[h]
\centering
\includegraphics[scale=1]{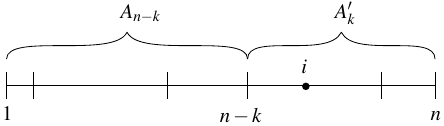}
\caption{Marked objects
of size $n$ ($nA_n$ many) 
can be split into an object 
of size $n-k$ ($A_{n-k}$ many) and an object of size $k$
with a mark in its first irreducible part ($A_k'$ many). 
Vertical lines delimit irreducible parts. 
Recurrence 
\eqref{E_ren_rec} 
follows, summing over $k$.
}
\label{F_Aprime}
\end{figure}

Lemma \ref{lem:renewal_mark} and the observation that $\Phi_n$ 
is a renewal sequence imply the following lemma 
(recalling that renewal times when
$S_{4k}=A_{4k}=0$ 
are only possible at multiples of $4$). 

\begin{lemma}
Let $\Phi'_n$ denote the number of pairs $(\cB,s)$, where $\cB$ is a 
Sina\u{\i} excursion of length $4n$ and $0\le s < \ell $ is an integer, 
where $\ell$ the length of the first irreducible Sina\u{\i} excursion of $\cB$. Then,  
\[n\Phi_n = \sum_{k=1}^n \Phi_k' \Phi_{n-k}. \]
\end{lemma}

Finally, to prove \eqref{E_PhiXi_rec} and thereby complete the proof of
Theorem \ref{T_SAforSE}, 
we show the following.

\begin{lemma}\label{L_XiX}
For all $n\ge1$, we have that $\Phi_n'=\Xi_n$. 
\end{lemma}

\begin{proof}
In fact, it will be easier to show that 
$2\Phi'_n=2\Xi_n$. By 
Lemma \ref{L_nor3n},
this is 
the number of subsets of $\{1,2,\ldots,4n-1\}$ of size $2n$
that sum to $n$ or $3n$ mod $4n$. 
To do this, we will find a bijection $\Upsilon$
\begin{itemize}
\item from the set of all pairs $(\cB,j)$, where $\cB$ is a 
Sina\u{\i} excursion of length $4n$, with first 
positive Sina\u{\i} excursion
of length $\ell=4k$, and $1\le j\le 4k$ is an integer, 
\item to the set of all of subsets $T$ 
of $\{1,2,\ldots,4n-1\}$
of size $2n$ 
that sum to $n$ or $3n$ mod $4n$. 
\end{itemize}

To describe $\Upsilon$, 
consider  $\cB$ as above. 
Let 
\[
0=i_1<\cdots< i_{2k}= 4k-1
\]
be the times before {\it up} steps in the first positive 
Sina\u{\i} excursion of $\cB$. 
Let $\cB^{(j)}$ be the bridge obtained from $\cB$ 
by cyclically shifting 
$\cB$ to the left by $i_j$. 
In other words, the $k$th increment of $\cB^{(j)}$ is
the $(k+i_j)$th (understood mod $4n$) increment of $\cB$. 
In particular, $\cB^{(1)}=\cB$. 
We let 
$\Upsilon(\cB,j)$ to be the set of 
times before {\it down} steps
in $\cB^{(j)}$. 
Lemma \ref{L_area} implies that, because  $\cB$ is a  
Sina\u{\i} excursion, its times before down steps
sum to $n(4n-1) \equiv 3n$ mod $4n$.
We shift all of these $2n$ times 
by the same amount
to obtain the times before down steps in $\cB^{(j)}$, 
so these sum to $n$ or $3n$ mod $4n$. 
See Figure \ref{F_bij} for an example. 

Finally, let us describe $\Upsilon^{-1}$. 
Let $T$ be as above. Consider the 
bridge $\cX$ with times before down steps 
at times 
$t\in T$. 
Then, by Lemma \ref{L_area}, 
the total area $A$ of $\cX$ is equal to $0$
mod $4n$. If we translate the $x$-axis
by some $\delta\in\Z$
the area of $\cX$, with respect to this 
new axis, is $A'=A-4\delta n$. 
Select the unique $\delta$ that sets $A'=0$. 
To find $\Upsilon^{-1}(T)$,
we choose the rightmost point 
before an up step 
in $\cX$ along this new axis 
for which the corresponding cyclic shift forms a Sina\u{\i} excursion
(with respect to this new axis). 
Since the total area is 0, such a point 
exists by 
Raney's lemma \cite{Ran60}. 
See Figure \ref{F_inv} for an example.
If this point occurs at time $m$ we set $\cB_i=\cX_{i+m}$ 
(with indices modulo $4n$) and we let $j$ be the index 
such that the $j$th up step in $\cB$ occurs at time $4n-m$.
\end{proof}

\begin{figure}[h]
\centering
\includegraphics[scale=0.8]{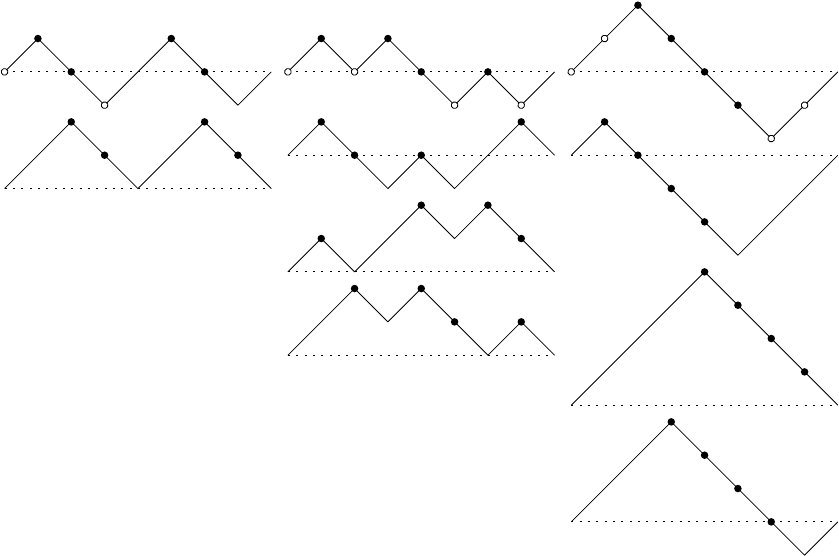}
\caption{{\it 1st row:} 
Times before down steps 
1256, 1346 and 2345
are solid dots 
in the $\Phi_2=3$ 
Sina\u{\i} excursions of length $8$. 
Times before up steps 
03, 0257 and 0167 
in their first positive 
excursions are open dots. 
The bijection $\Upsilon$ gives the $2\Xi_2=10$
subsets of 
$\{1,2,\ldots,7\}$
of size $4$ that
sum to $2$ or $6$ mod $8$ as follows. 
{\it 1st column:} 
If we shift the first Sina\u{\i} excursion 
by starting at its 1st or 2nd 
open dot, we obtain bridges with times
before down steps (solid dots) at
times 1256 and 2367. 
{\it 2nd column:} 
If we shift the 2nd Sina\u{\i} excursion 
by starting at its 1st, 2nd, 3rd or 4th
open dot, we obtain bridges with times
before down steps 
at 1346, 1247, 1467 and 2457. 
{\it 3rd column:} 
If we shift the 3nd Sina\u{\i} excursion 
by starting at its 1st, 2nd, 3rd or 4th
open dot, we obtain bridges with times
before down steps
at 2345, 1234, 4567 and 3456. 
}
\label{F_bij}
\end{figure}

\begin{figure}[h]
\centering
\includegraphics[scale=0.8]{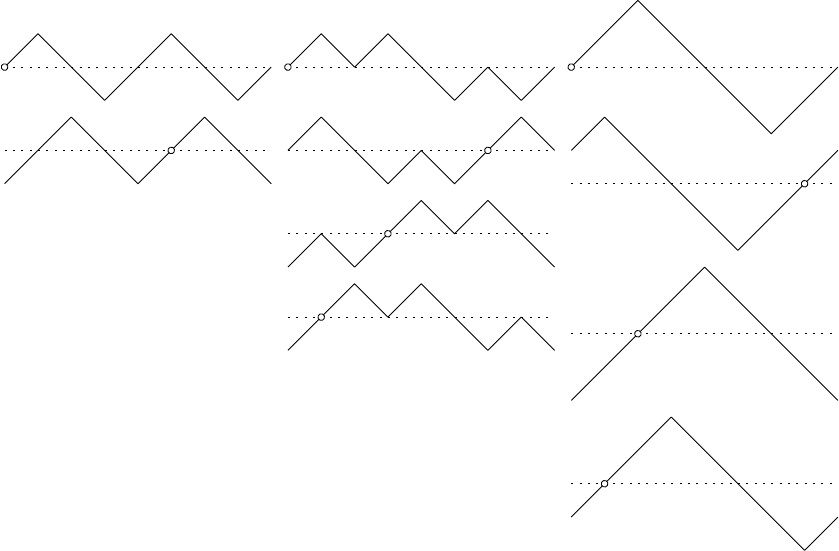}
\caption{To invert the bijection $\Upsilon$, 
we select the horizontal line (dotted) 
that ``cuts''
the area in half, and use Raney's lemma
to find the rightmost point (open dot) 
before an up step 
that starts a Sina\u{\i} excursion, 
with respect to this line. 
}
\label{F_inv}
\end{figure}

\section{A related application}

Finally, let us conclude with 
another, related application of
our current techniques. 

As a consequence of considerably more 
technical arguments
than those in the current article, 
we \cite[Corollary 4]{DK24a} 
recently 
proved that 
\begin{equation}\label{E_pn}
n^{1/2}p_n
\to
\frac{1}{2}\sqrt{\frac{\pi}{6}}
\frac{1}{1-\P(A_\tau=0)}, 
\end{equation}
where $\tau=\inf\{t:S_t=0,\: A_t\le 0\}$. 
Informally, $\tau$ is the first time that a random walk 
is ``at risk'' of not being Sina\u{\i} (i.e., if $A_\tau<0$).  
Theorem \ref{T_main} above implies that  
$\P(A_\tau=0)
=1-e^{-\lambda}$, 
where $\lambda$ is as in \eqref{E_lam}, 
so that 
\[
n^{1/2}p_n
\to 
\frac{1}{2}\sqrt{\frac{\pi}{6}}
e^\lambda.
\]

This, in turn, when combined with Proposition 6.3 in \cite{BDGJS22}, 
yields the following corollary,
concerning the 
probability that a random walk bridge 
of length $2n$ is a 
Sina\u{\i} walk.

\begin{corollary}
\label{cor:meander}
As $n\to\infty$,  it holds that
\[n^{1/4}\P(A_1,\ldots,A_{2n}\ge0\mid S_{2n}=0)
\to \frac{e^{\lambda/2}\sqrt{\pi}}{\Gamma(1/4)}.  
\]
\end{corollary}

\section{Acknowledgments}
SD acknowledges
the financial support of the CogniGron research center
and the Ubbo Emmius Funds (University of Groningen).
BK was partially supported by a 
Florence Nightingale Bicentennial Fellowship (Oxford Statistics)
and a Senior Demyship (Magdalen College).
BK gratefully acknowledges 
Vasu Tewari for indicating  
the work of von Sterneck.

\makeatletter
\renewcommand\@biblabel[1]{#1.}
\makeatother

\providecommand{\bysame}{\leavevmode\hbox to3em{\hrulefill}\thinspace}
\providecommand{\MR}{\relax\ifhmode\unskip\space\fi MR }
\providecommand{\MRhref}[2]{%
  \href{http://www.ams.org/mathscinet-getitem?mr=#1}{#2}
}
\providecommand{\href}[2]{#2}


\begin{thebibliography}{10}

\bibitem{ABR08}
L.~Addario-Berry and B.~A.~Reed, \emph{Ballot theorems, old and new},
  Horizons of combinatorics, Bolyai Soc. Math. Stud., vol. 17, Springer, Berlin,
  2008, pp.~9--35.


\bibitem{ADL14}
F.~Aurzada, S.~Dereich and M.~Lifshits, \emph{Persistence probabilities for a
  bridge of an integrated simple random walk}, Probab. Math. Statist.
  \textbf{34} (2014), no.~1, 1--22.

\bibitem{AS15}
F.~Aurzada and T.~Simon, \emph{Persistence probabilities and exponents},
  L\'{e}vy matters. {V}, Lecture Notes in Math., vol. 2149, Springer, Cham,
  2015, pp.~183--224.

\bibitem{Bac02}
P.~G.~H. Bachmann, \emph{Niedere {Z}ahlentheorie}, vol.~2, Leipzig B.G.
  Teubner, 1902.

\bibitem{BDGJS22}
P.~Balister, S.~Donderwinkel, C.~Groenland, T.~Johnston, and A.~Scott,
 \emph{Counting graphic sequences via integrated random walks}, Trans. Amer. Math. Soc.
  \textbf{378} (2025), no.~7, 4627--4669.

\bibitem{BDW23}
M.~B\"{a}r, J.~Duraj, and V.~Wachtel, \emph{Invariance principles for
  integrated random walks conditioned to stay positive}, Ann. Appl. Probab.
  \textbf{33} (2023), no.~1, 127--160.
  
\bibitem{BDK25}
M.~Bassan, S.~Donderwinkel and B.~Kolesnik,
  \emph{Tournament score sequences,  
Erd\H{o}s--Ginzburg--Ziv numbers,
and the L\'evy--Khintchine method}, 
preprint (2025), available at
  \href{https://arxiv.org/abs/2407.01441}{arXiv:2407.01441}.
  
\bibitem{Ber98}
J.~Bertoin, \emph{The inviscid Burgers equation with Brownian initial velocity}, Comm. Math. Phys.
  \textbf{193} (1998), no.~3, 397--406.

\bibitem{Bol00}
E.~Bolthausen, \emph{Random walk representations and entropic repulsion for
              gradient models},
  Infinite dimensional stochastic analysis, 
  Verh. Afd. Natuurkd. 1. Reeks. K. Ned. Akad. Wet., vol. 52, R. Neth. Acad. Arts Sci., Amsterdam,
  2000, pp.~55--83.



\bibitem{BMS13}
A.~J.~Bray, S.~N.~Majumdar and G.~Schehr, \emph{Persistence and first-passage properties in nonequilibrium systems}, Adv. Phys.
  \textbf{62} (2013), no.~3, 225--361.

  

\bibitem{CD08}
F.~Caravenna and J.-D. Deuschel, \emph{Pinning and wetting transition for
  {$(1+1)$}-dimensional fields with {L}aplacian interaction}, Ann. Probab.
  \textbf{36} (2008), no.~6, 2388--2433.
  
\bibitem{CNW73}
J.~Chover, P.~Ney and S.~Wainger, \emph{Functions of probability measures}, J. Analyse Math. \textbf{26} (1973), 255--302.

\bibitem{CDPCMDS04}
M.~Constantin, C.~Dasgupta, P.~Punyindu Chatraphorn, S.~N.~Majumdar and S.~Das Sarma, \emph{Persistence
in nonequilibrium surface growth}, Phys. Rev. E \textbf{69} (2004), 22.


\bibitem{BDW23}
A.~Dembo, J.~Ding and F.~Gao, \emph{Persistence of iterated partial sums}, Ann. Inst. Henri Poincar\'{e} Probab. Stat.
  \textbf{49} (2013), no.~3, 873--884.

\bibitem{DK24a}
S.~Donderwinkel and B.~Kolesnik, \emph{Tournaments and random walks}, preprint (2024), 
  available at \href{https://arxiv.org/abs/2403.12940}{arXiv:2403.12940}.

\bibitem{EH82}
P.~Embrechts and J.~Hawkes, \emph{A limit theorem for the tails of discrete
  infinitely divisible laws with applications to fluctuation theory}, J.
  Austral. Math. Soc. Ser. A \textbf{32} (1982), no.~3, 412--422.
  
\bibitem{EO84}
P.~Embrechts and E.~Omey, \emph{Functions of power series}, Yokohama Math. J. \textbf{32} (1984), no.~1-2, 77--88.

\bibitem{F68}
W.~Feller, \emph{An introduction to probability theory and its applications. Vol. I}, third ed.,
  John Wiley \& Sons, Inc., New York-London-Sydney, 1968

\bibitem{BDW23}
F.~Gao, Z.~Liu and X.~Yang, \emph{Conditional persistence of {G}aussian random walks}, Electron. Commun. Probab.
  \textbf{19} (2014), no.~70, 9.

\bibitem{Gia07}
G.~Giacomin, \emph{Random polymer models},
  Imperial College Press, London, 
  2007, xvi+242.

\bibitem{Han93}
K.~Handa, \emph{A remark on shocks in inviscid Burgers’ turbulence}, In Fitzmaurice et al. (ed): Nonlinear
waves and weak turbulence with applications in oceanography and condensed matter physics. Prog.
Nonlinear Diﬀer. Equ. Appl. \textbf{11} (1993), 339--345.

\bibitem{HJ78}
J.~Hawkes and J.~D. Jenkins, \emph{Infinitely divisible sequences}, Scand.
  Actuar. J. (1978), no.~2, 65--76.

\bibitem{Maj99}
S.~N.~Majumdar, \emph{Persistence in nonequilibrium systems}, Current Science \textbf{77} (1999), no.~3, 370--375.  

\bibitem{Ram44}
K.~G. Ramanathan, \emph{Some applications of {R}amanujan's trigonometrical sum
  {$C_m(n)$}}, Proc. Indian Acad. Sci., Sect. A. \textbf{20} (1944), 62--69.

\bibitem{Ran60}
G.~N. Raney, \emph{Functional composition patterns and power series reversion},
  Trans. Amer. Math. Soc. \textbf{94} (1960), 441--451.

\bibitem{Sin92}
Ya.~G. Sina\u{\i}, \emph{Distribution of some functionals of the integral of a
  random walk}, Teoret. Mat. Fiz. \textbf{90} (1992), no.~3, 323--353.
  
\bibitem{Sin92b}
\bysame, \emph{Statistics of shocks in solutions of inviscid Burgers equation}, Comm. Math. Phys. \textbf{148} (1992), no.~3, 601--621.

\bibitem{SpaAnd54}
E.~Sparre~Andersen, \emph{On the fluctuations of sums of random variables.
  {II}}, Math. Scand. \textbf{2} (1954), 195--223.

\bibitem{Vel06}
Y.~Velenik, \emph{Localization and delocalization of random interfaces},
  Probab. Surv. \textbf{3} (2006), 112--169.

\bibitem{Vys14}
V.~Vysotsky, \emph{Positivity of integrated random walks}, Ann. Inst. Henri
  Poincar\'{e} Probab. Stat. \textbf{50} (2014), no.~1, 195--213.

\end{thebibliography}
\end{document}